\documentclass{amsart}
\newcommand{\CA}{{\mathcal A}}
\newcommand{\CB}{{\mathcal B}}
\newcommand{\CC}{{\mathcal C}}

\newcommand{\CF}{{\mathcal F}}
\newcommand{\CG}{{\mathcal G}}

\newcommand{\CM}{{\mathcal M}}

\newcommand{\CP}{{\mathcal P}}
\newcommand{\CQ}{{\mathcal Q}}

\newcommand{\baton}[1]{{\mathbb #1}}
\newcommand{\N}{{\baton N}}

\newcommand{\Z}{{\baton Z}}

\newcommand{\E}{{\baton E}}

\newcommand{\one}{\mbox{\rm 1 \hspace{-.58em}l}}

\DeclareMathOperator{\supp}{Supp}
\newcommand{\egaldef}{\stackrel{\text{def}}{=}}
\DeclareMathOperator*{\Ttimes}{\times}
\newcommand{\ttimes}{\displaystyle \Ttimes}

\newtheorem{proposition}{Proposition}
\newtheorem{corollary}{Corollary}
\newtheorem{lemma}{Lemma}

\newtheorem{definition}{Definition}
\newtheorem*{VP}{Variational Principle}
\newtheorem*{MT}{Martingale Theorem}
\newtheorem*{PF}{Pinsker Formula}

\theoremstyle{remark}
\newtheorem{remark}{Remark}

\begin{document}

\title{ Asymptotic pairs in positive-entropy systems.}

\author{F. Blanchard \and  B. Host\and  S. Ruette}

%\date{March 20, 2001}
\thanks{Ergod. Th. \& Dynam. Syst., {\bf 22}, 671-686, 2002.}
\address{Fran\c cois Blanchard, Institut de Math\'ematiques de Luminy -
CNRS UPR 9016 - 163, avenue de Luminy, case 907 -
13288 Marseille cedex 9 - France}
\email{blanchar@iml.univ-mrs.fr}
\address{Bernard Host, Universit\'e de Marne la Vall\'ee -
Cit\'e Descartes, 5, boulevard Descartes -
Champs sur Marne, 77454 Marne-la-Vall\'ee Cedex 2 - France}
\email{host@math.univ-mlv.fr}
\address{Sylvie Ruette, Institut de Math\'ematiques de Luminy -
CNRS UPR 9016 - 163, avenue de Luminy, case 907 -
13288 Marseille cedex 9 - France}
\email{ruette@iml.univ-mrs.fr}

\begin{abstract}
We show that in a topological dynamical system $(X,T)$ of positive
entropy there exist proper (positively) asymptotic pairs, that is, pairs
$(x,y)$
such that $x\not= y$ and $\lim_{n\to +\infty} d(T^n x,T^n y)=0$.  More
precisely we consider a
$T$-ergodic measure $\mu$ of positive entropy and prove that the
set of points that belong to a proper asymptotic pair is of measure
$1$.  When $T$ is invertible, the stable classes (i.e., the equivalence
classes for the asymptotic equivalence) are not stable under $T^{-1}$: for
$\mu$-almost every $x$ there are uncountably many $y$ that are asymptotic with
$x$ and such that $(x,y)$ is a Li-Yorke pair with respect to $T^{-1}$.  We also
show that asymptotic pairs are dense in the set of topological entropy pairs.
\end{abstract}

\maketitle

\section{Introduction}
              \label{sec:intro}

In this article a \emph{topological dynamical system} is
a compact metric space $X$ endowed with a homeomorphism $T:X\to X$, except in
subsection 3.3 where we drop the assumption that $T$ is invertible; the
distance
on $X$ is denoted by $d$.

Classically in Topological Dynamics one considers the asymptotic behaviour
of pairs of points. In this article, even when the systems considered are
invertible, the definitions of asymptoticity, proximality and Li-Yorke pairs
that we use are those fitted to an $\N$-action. A pair $(x,y)\in
X\times X$ is said to be \emph{proximal} if
$\liminf_{n\to+\infty}d(T^nx,T^ny)=0$, and $(x,y)$ is called an
\emph{asymptotic
pair} if $\lim_{n\to+\infty}d(T^nx,T^ny)=0$; the set of asymptotic pairs is
denoted by ${\mathbf A}$. An asymptotic pair $(x,y)$ with $x\neq y$ is said
to be
\emph{proper}. Asymptoticity is an equivalence relation; the
equivalence class of a point is called its
\emph{stable class}. We call a proximal pair
that is not asymptotic a \emph{Li-Yorke pair}: in 1975 Li and Yorke introduced
such pairs in a tentative definition of chaos \cite{LY}.

It is proven in \cite{BGKM} that positive entropy implies the existence
of a topologically `big' set of Li-Yorke pairs.
Here we prove by ergodic methods that in any topological dynamical system with
positive topological entropy there is a measure-theoretically
`rather big' set of proper asymptotic pairs; this is obvious for a
symbolic system but not in general. The set of
asymptotic pairs of any topological dynamical system has
been shown to be first category in~\cite{HY2}: it is a small set, but not too
small according to the present result. We also show
that a `rather big' set of $T$-asymptotic pairs are Li-Yorke under the
action of $T^{-1}$.

In~\cite{HY1} Huang and Ye construct a completely scrambled system, that is
to say, a dynamical system $(X,T)$ such that all proper pairs in $(X\times
X)$ are Li-Yorke.  They ask whether
such a system may have positive entropy.  That it may not is a direct
consequence of our Proposition 1. This statement formally generalizes a
previous
result of Weiss~\cite{We}, showing that any system $(X,T)$ such that
$(X\times X,T\times T)$ is recurrent has entropy $0$; recurrence of
$(X\times X,T\times T)$ means that any pair $(x,y)$, $x\ne y$, comes back
arbitrarily close to itself under the action of powers of $T$, which implies
that it cannot be asymptotic.

Then we study the behaviour of $T$-asymptotic pairs under $T^{-1}$.
Anosov diffeomorphisms on a manifold have stable and unstable
foliations; points belonging to the same stable foliation are
asymptotic under $T$ and tend to diverge under $T^{-1}$, while pairs
belonging to the unstable foliation behave the opposite way.  Our
results show that any positive-entropy system retains a faint flavour
of this situation:
there is a universal $\delta>0$ such that
outside a `small' set  the stable class of $x$ is non-empty and contains
an uncountable set of points $y$ such that
$\limsup_{n\to+\infty} d(T^{-n}x,T^{-n}y)\geq \delta$.

We also obtain a result about entropy pairs~\cite{BHM}: the set of
asymptotic pairs ${\mathbf A}$ is dense in the set of  entropy pairs $E(X,T)$.
The proof relies on two facts: that
the union of the sets of $\mu$-entropy pairs for all ergodic measures
$\mu$ is dense in the set of topological entropy pairs~\cite{BGH}, and the
characterization of the set $E_\mu(X,T)$ of $\mu$-entropy pairs as the  support
of some measure on $X\times X$~\cite{G}.

The article is organized as follows.  Section~\ref{sec:back} contains
some background in Ergodic Theory, in particular the old but not very
familiar definition of an
\emph{excellent partition}. In Section~\ref{sec:existence} using an
ad-hoc  excellent partition  we show that every system of positive entropy
admits `many' asymptotic pairs, and that this is also true for
non-invertible systems. In the next section, after recalling the
definition of the relative independent square of a measure, we use
this notion to show that asymptotic pairs are dense in the set of entropy
pairs.
In Section~\ref{sec:Li-Yorke}, we show that a system of positive
entropy has  `many' pairs that are asymptotic for $T$ and Li-Yorke
for $T^{-1}$. In the last section we show that the sets constructed
above are uncountable.

Some results are stated several times in increasingly
strong form; Propositions~\ref{prop:bizarre} and ~\ref{prop:ouf} are strongest.
We chose this organization in order to avoid a long preliminary
section containing all the required background.
Most tools are introduced just before the statements that require
them for their proofs.

A final remark about the methods. It is not very satisfactory to prove a purely
topological result -- the existence of many asymptotic pairs in any
positive-entropy topological dynamical system -- in a purely ergodic way.
Proving
it topologically is a good challenge. On the other hand Ergodic Theory is a
powerful tool; it is not the first time that it demonstrates its strength
in a neighbouring field. Here it also permits to prove results that are
probabilistic in nature.

We are grateful to X.D. Ye for providing the initial motivation, to
W.  Huang and him for several valuable observations and to S. Kolyada for
various interesting remarks. The referee made significant remarks
and corrected many English mistakes.

\section{Background}
              \label{sec:back}

Here are some classical definitions and results from Ergodic Theory, and
some technical Lemmas that will be needed in the sequel.

A measure-theoretic dynamical system $(X,\CA, T,\mu)$
is a Lebesgue probability space $(X,\CA,\mu)$ endowed with a
measurable transformation $T\colon X\to X$ which preserves $\mu$.
In this article unless
stated otherwise $T$ is assumed to be one-to-one and bi-measurable.
The $\sigma$-algebra $\CA$ is assumed to be complete for $\mu$. 
All measures are assumed to be probability measures;
since quasi-invariant measures are not considered in this article,
an ergodic measure is always assumed to be invariant. 

\subsection{Partitions}
                  \label{subsec:partitions}

All partitions of $X$ are assumed to consist
of atoms belonging to the $\sigma$-algebra $\CA$.
Given a partition $\CP$ of $X$ and $x\in X$, denote by $\CP(x)$
the atom of $\CP$ containing $x$.

If $(\CP_i;i\in I)$ is a
countable  family of finite partitions, the partition
$\CP=\bigvee_{i\in I}\CP_i$ is called a \emph{measurable
partition}~\cite{P1}.
The sets $A\in\CA$ which are union of atoms of $\CP$ form a
sub-$\sigma$-algebra of $\CA$ denoted by $\sigma(\CP)$ or $\CP$
if there is no ambiguity.
% For each set $A\in\CA$, the (actually countable) union $\CP(A)=\bigcup_{x\in
% A}\CP(x)$ belongs to $\CA$ too; these sets form a sub-$\sigma$-algebra
% of $\CA$ denoted by $\sigma(\CP)$, or $\CP$ if there is no
% ambiguity.
Every sub-$\sigma$-algebra
of $\CA$ coincides with a $\sigma$-algebra constructed in this way
outside a set of measure $0$.

A sub-$\sigma$-algebra $\CF$
of $\CA$ which is $T$-invariant, that is, $T^{-1}\CF=\CF$, is called a
\emph{factor}. Equivalently, a factor is given by a measure-theoretical system
$(Y,\CB,S,\nu)$ and a measurable map $\varphi\colon X\to Y$ such that
$\varphi\circ T=S\circ \varphi$; the corresponding $T$-invariant
sub-$\sigma$-algebra of $\CA$ is $\varphi^{-1} \CB$.

  Given a measurable partition $\CP$, put
$\CP^-= \bigvee_{n=1}^{\infty} T^{-n}\CP$ and $\CP^T=
\bigvee_{n=-\infty}^{+\infty} T^{-n}\CP$.
  Define in the same way $\CF^-$ and
$\CF^T$ if $\CF$ is a sub-$\sigma$-algebra of $\CA$.
The measurable partition $\CP$ (resp. the sub-$\sigma$-algebra $\CF$) is
called \emph{generating} if $\sigma(\CP^T)$  (resp. $\CF^T$) is equal to
$\CA$.

\subsection{Entropy}
             \label{subsec:entropy}

For the definition of the conditional entropy $H(\CP \mid \CF)$
of a finite measurable partition  $\CP$ with respect to the
sub-$\sigma$-algebra $\CF$, of the entropy $h_\mu(\CP,T)=H(\CP \mid \CP^-)$
of  a partition $\CP$ with respect to $T$ and of the entropy
$h_{\mu}(X,T)$, refer to~\cite{P1}, \cite{P2}, \cite{Wa}.

The \emph{Pinsker factor } $\Pi_{\mu}$ of $(X,\CA, T,\mu)$ is the maximal
factor with entropy $0$; a finite partition $\CP$ is measurable with
respect to $\Pi_{\mu}$ if and only if
$h_{\mu}({\CP},T)=0$.

We do not give the proofs of the next two results; they can be found in
\cite{P1}.

\begin{lemma}\label{lemma:generating}
If $\CF$ is a generating sub-$\sigma$-algebra then
$\displaystyle \Pi_{\mu} \subset\CF^-$.
\end{lemma}

\begin{PF}
For any finite partitions $\CP$ and $\CQ$ one has
\begin{equation}
        \label{eq:pinsker}
%h_\mu(\CQ^T\mid  \CP^T)\egaldef
H({\CQ}\vee\CP\mid  \CQ^-\vee \CP^-) -H(\CP\mid\CP^-)
= H(\CQ\mid\CQ^-\vee \CP^T)\ .
\end{equation}
\end{PF}

The next technical Lemma compares the entropy of a partition
with the conditional entropy of this partition with respect
to the past of another.

\begin{lemma}
               \label{lemma:technical}
Let $(X,\CA, T,\mu)$ be a
measure-theoretic dynamical system, and let
$\CP_1\prec\CP_2\prec \dots\prec\CP_k$ be
finite partitions. Then
\begin{equation}
      \label{eq:diff-entropy}
H(\CP_1\mid \CP_1^-) - H(\CP_1\mid \CP_2^-)=
H(\CP_2\mid \CP_1\vee \CP_2^-) - H(\CP_2\mid
\CP_1^T\vee \CP_2^-)
\end{equation}
and
\begin{equation}
     \label{eq:sum-entropy}
H(\CP_1\mid \CP_1^-) - H(\CP_1\mid \CP_k^-)
\leq \sum_{i=1}^{k-1}
\Bigl( H(\CP_i\mid \CP_i^-) - H(\CP_i\mid\CP_{i+1}^-)\Bigr)\ .
\end{equation}
\end{lemma}

\begin{proof}
Obviously $\CP_k= \CP_1\vee\dots \vee \CP_k$.
A repeated use of the Pinsker Formula~\eqref{eq:pinsker}
yields
$$
H(\CP_k \mid \CP_k^-) =
H(\CP_1\mid \CP_1^-) + H(\CP_2\mid\CP_2^-\vee \CP_1^T)
+\dots+  H(\CP_k\mid \CP_k^-\vee \CP_{k-1}^T)\ ;
$$
also, using the elementary formula for conditional
entropy of partitions inductively one gets
$$
H(\CP_k \mid \CP_k^-) = H(\CP_1\mid \CP_k^-) +
H(\CP_2\mid \CP_k^-\vee \CP_1)+\dots+
H(\CP_k\mid \CP_k^-\vee \CP_{k-1})\ .
$$
Combining these two equalities one obtains
$$
H(\CP_1\mid \CP_1^-) - H(\CP_1\mid  \CP_k^-)=
\sum_{i=2}^k \bigl( H(\CP_i\mid \CP_k^-\vee\CP_{i-1})
- H(\CP_i\mid \CP_i^-\vee \CP_{i-1}^T) \bigr)\ .
$$

For $k=2$ this is \eqref{eq:diff-entropy}.

For $k>2$, remark that $\CP_i^-\prec \CP_k^-$ for $i\leq k$ so that
$H(\CP_i\mid\CP_k^-\vee \CP_{i-1})
\leq H(\CP_i\mid \CP_i^-\vee \CP_{i-1})$,
  hence
$$
H(\CP _1\mid \CP _1^-) - H(\CP _1\mid \CP _k^-)
  \leq   \sum_{i=2}^{k}
  \bigl(H(\CP _i\mid \CP _i^-\vee \CP_{i-1})
        - H(\CP_i\mid \CP_i^-\vee \CP _{i-1}^T) \bigr)\ .
$$
Applying \eqref{eq:diff-entropy}
(with $\CP_{i-1}$ and $\CP_i$  in place of $\CP_1$ and $\CP_2$)
  to each term in the sum, the inequality above becomes
$$
H(\CP_1\mid \CP_1^-) - H(\CP_1\mid \CP_k^-)
\leq  \sum_{i=2}^{k} \bigl(
     H(\CP_{i-1}\mid \CP_{i-1}^-)-H(\CP_{i-1}\mid \CP_i^-)
\bigr)\,
$$
which is \eqref{eq:sum-entropy} up to a
change of index.
\end{proof}

\subsection{Excellent partitions}
                  \label{subsec:excellent}

For any measure-theoretic dynamical system $(X,\CA, T,\mu)$ there exists a
generating measurable partition
with the property that $\displaystyle\bigcap_{k=1}^\infty T^{-k}\CP ^-=
\Pi_{\mu}$. In the finite-entropy case any finite generating
partition has this property. The existence of such a partition in the
general case was proven by Rohlin and Sina\u{\i} and permitted to show that
the class of K-systems and the class of completely positive entropy systems
coincide \cite{RS}; they gave a construction from which the one in
Subsection \ref{subsec:construction} is derived.
The name ``excellent'' was coined by one of the present authors in a later
article.
\begin{definition}
                  \label{def:excellent}
Let $(X,\CA,T, \mu)$ be a measure-theoretic dynamical system.  A
measurable partition $\CP $ is said to be \emph{excellent} if it is
generating and there is an increasing sequence of finite measurable
partitions $(\CP_n)_{n\geq 1}$ such that $\CP_n \to \CP$ and
$H(\CP_n\mid \CP_n^-)- H(\CP_n \mid \CP^-) \to 0$ as $n\to \infty$.
\end{definition}

\begin{lemma}
     \label{lemma:excellent}~\cite{P1}
If $\CP $  is an excellent partition,   then
$\displaystyle\bigcap_{k=1}^\infty T^{-k}\CP ^-= \Pi_{\mu}$.
\end{lemma}

\begin{proof}
Let $\CQ$ be a  finite partition,  measurable  with respect to
$\bigcap_{k=1}^{\infty} T^{-k}\CP ^-$,
  and let the partitions $\CP _n$ be as in Definition~\ref{def:excellent}.
Applying the Pinsker formula twice one obtains
\begin{align*}
H(\CQ \mid \CQ^-)
   & = H(\CP_n\vee \CQ\mid \CP_n^- \vee \CQ ^-)
       - H(\CP_n\mid \CP_n^-\vee\CQ^T)\\
   & = H(\CP_n\mid \CP_n^-) +H(\CQ\mid\CP_n^T\vee \CQ^-)
       - H(\CP_n\mid \CP_n^-\vee\CQ ^T)\ .
\end{align*}
When $n$ goes to infinity $H(\CQ\mid\CP_n^T \vee \CQ ^-)$ goes to $0$,
since $\CP_n^T$ tends to $\CP^T=\CA$;
on the other hand we assumed
that $T^n\CQ$ is measurable with respect to $\CP^-$ for $n\in \Z$,
  so $\CP_n^-\vee\CQ^T$ is contained in $\CP^-$ and
$$
0\leq H(\CP_n\mid\CP_n^-) -  H(\CP_n\mid \CP_n^-\vee \CQ^T)
\leq  H(\CP_n \mid \CP_n^-) - H(\CP_n \mid \CP ^-)\ .
$$
By our assumption the majoration tends to
$0$.
Thus $H(\CQ\mid\CQ^-) =0$, which means that
$\CQ$ is coarser than the Pinsker $\sigma$-algebra.
As this is true for any finite partition $\CQ$ measurable with respect
to  $\bigcap_{k=1}^{\infty} T^{-k}\CP ^-$, one has
$\bigcap_{k=1}^{\infty}T^{-k}\CP ^-\subset\Pi_\mu$.

The reverse inclusion is due to the fact that $\CP $ is generating,
so that
$\Pi_{\mu} \subset \bigcap_{k=1}^{\infty}T^{-k}\CP ^-$ by Lemma
\ref{lemma:generating}.
\end{proof}

\section{existence of asymptotic pairs}
                                     \label{sec:existence}

Let $(X,T)$ be a topological dynamical system,
and let $\CB$ be the Borel $\sigma$-algebra of $X$. Given two topological
dynamical systems $(X,T)$ and $(Y,S)$ a continuous onto map
$\pi:(X,T)\to(Y,S)$ such that $\pi\circ T=S\circ \pi$ is called  a topological
factor map.

The
definitions of proximal, asymptotic and Li-Yorke pairs  are given
at the very beginning of the Introduction. Recall that $\mathbf A$ is the
set of
all asymptotic pairs in $X\times X$.
See~\cite{Wa} for the definition of topological entropy, and for
the

\begin{VP} The topological entropy $h(X,T)$ of the system $(X,T)$ is
equal to the supremum of the entropies $h_\mu(X,\CB,T,\mu)$ where $\mu$ ranges
over the set of ergodic $T$-invariant measures.
\end{VP}

\subsection{Construction of an excellent partition.}\label{subsec:construction}

The next Lemma establishes a connection
between asymptotic pairs and entropy. It is our main tool. It is based
on the construction of excellent partitions in~\cite{P1}.

\begin{lemma}
     \label{lemma:constr-excellent}
     Let $\mu$ be an ergodic measure on $X$.

$\imath$) The system $(X,\CB,T,\mu)$ admits an excellent partition $\CP$, such
that any pair of points belonging to the same atom of $\CP^-$ is
asymptotic.

$\imath\imath$) Moreover, if $h_\mu(X,T)>0$ then the $\sigma$-algebras
$\CP^-$ and
$\CB$ do not coincide up to sets of $\mu$-measure $0$.
\end{lemma}

\begin{proof}

$\imath$)
Let $(\CQ_n)_{n\geq 1}$ be an increasing sequence of finite partitions
such that the maximal diameter $\delta_n$ of an element of $\CQ_n$
goes to $0$ as $n\to\infty$, and $(\epsilon_n)_{n\geq 1}$ be a
sequence of positive numbers such that
$\sum_{i=1}^{\infty} \epsilon_n < \infty$.

We construct inductively an increasing sequence $(k_n)_{n\geq 1}$ of
non-negative  integers such that, if
$$
\CP_i=T^{-k_1}\CQ_1\vee T^{-k_2}\CQ_2\dots\vee T^{-k_i}\CQ_i
$$
for $ i\geq 1$ one has
\begin{equation}
     \label{eq:rec}
H(\CP_i\mid\CP_i^-) - H(\CP_i\mid \CP_{i+1}^-) < \epsilon_i\ .
\end{equation}

Put $k_1=0$ and $\CP_1 = \CQ_1$. Take $n\geq 2$, and
suppose that the sequence is already defined up to
$k_{n-1}$ and the bound~\eqref{eq:rec} holds for $1\leq i\leq n-2$.

By Lemma~\ref{lemma:technical}~\eqref{eq:diff-entropy} one has for
$k\geq 0$

\begin{align*}
D_k
&\egaldef H(\CP_{n-1}\mid \CP_{n-1}^-) -
H(\CP_{n-1}\mid \CP_{n-1}^-\vee T^{-k}\CQ_n^-)\\
&= H(T^{-k}\CQ_n\mid \CP_{n-1}\vee \CP_{n-1}^-\vee T^{-k}\CQ_n^-)
- H(T^{-k}\CQ_n \mid \CP_{n-1}^T\vee T^{-k}\CQ _n^-)\ .
\end{align*}
By $T$-invariance of $\mu$ the second equality above becomes
$$
D_k = H(\CQ_n\mid T^{k+1}\CP_{n-1}^-\vee \CQ_n^-)
-H(\CQ_n \mid \CP_{n-1}^T\vee \CQ_n^-)\ ;
$$
when $k$
goes to infinity the conditioning $\sigma$-algebra in the first term
tends to the conditioning $\sigma$-algebra in the second term, and
the difference $D_k$ tends to $0$.

Fix $k_n$ so that
$D_{k_n} < \epsilon_n$,
  which, putting $\CP_n=\CP_{n-1} \vee T^{-k_n}\CQ_n$,
is Property~\eqref{eq:rec} at rank $i=n-1$. Setting $\CP=\bigvee_{n\in
\N}\CP_n$
completes our construction.

It remains to check that $\CP$ is excellent.

By construction, $\CP^T$ is finer than  $\bigvee_{n\geq 1}\CQ_n$,
and this partition spans $\CB$ because of our hypotheses on $(\CQ_n)$.
Thus $\CP$ is generating.

The sequence $(\CP_n)$ increases to $\CP$; moreover
$$
H(\CP_n \mid \CP_n^-) - H(\CP_n \mid \CP^-)=
\lim_{k \to \infty}\bigl (H(\CP_n \mid \CP_n^-)
- H(\CP_n \mid \CP_{n+k}^-)\bigr )\ ,
$$
and by Lemma~\ref{lemma:technical}~\eqref{eq:sum-entropy} one gets
\begin{equation}
     \label{eq:exc}
H(\CP_n \mid \CP_n^-) - H(\CP_n \mid \CP^-)\leq
\sum_{i=n}^{\infty} \big( H(\CP_i\mid\CP_i^-) -
H(\CP_i\mid \CP_{i+1}^-)\big)
<  \sum_{i=n}^{\infty} \epsilon_i\ ,
\end{equation}
a quantity which vanishes as $n\to
\infty$: the second condition for excellence of $\CP$ holds.

Let $x,y$ belong to the same atom of $\CP^-$. For each $i\geq 1$,
$T^ix$ and $T^iy$  belong to the same atom of $\CP$, thus
$T^{i+k_n}x$ and $T^{i+k_n}y$ belong to the same atom of $\CQ_n$ for all
$n\geq 1$. For all $k>k_n$ the points $T^kx$ and $T^ky$ belong to the
same atom of $\CQ_n$, and $d(T^kx,T^ky)\leq\delta_n$, thus $x,y$ are
asymptotic.

$\imath\imath$)
Assume that $\CP^- =\CB$, then by \eqref{eq:exc} one obtains
$H(\CP_n|\CP_n^{-})\to 0$. In addition
$$
H(\CP_n|\CP_n^{-}) = h_{\mu}(\CP_n,T) \ge  h_{\mu}(\CQ_n,T)\to h_{\mu}(X,T), $$
so $h_{\mu}(X,T)=0$. This completes the proof.
\end{proof}

\subsection{The invertible case}
       \label{subsec:existence}
For $x \in X$ denote by ${\mathbf A}(x)$ the set of points of $X$ that are
asymptotic to $x$.
\begin{proposition}
                    \label{prop:existence}
Let $(X,T)$ be an invertible topological dynamical system with positive
topological entropy. Then $(X,T)$ has proper asymptotic pairs.

More precisely, the set of points belonging to a proper asymptotic
pair has measure $1$ for any ergodic measure on $X$ with
positive entropy.
\end{proposition}

\begin{proof} 
Let $\mu$ be an ergodic measure on $X$ with $h_\mu(X,T)>0$;
the existence of $\mu$ follows from the Variational Principle. Let
$\CP$ be the excellent partition for $(X,\CB,T,\mu)$ constructed in
Lemma~\ref{lemma:constr-excellent}.

Let $J$ be the set of points of $X$ which belong to a proper asymptotic
pair. $J$ is measurable and invariant under $T$.

By ergodicity $\mu(J)=0$ or $1$; assume that $\mu(J)=0$. Then
${\mathbf A}(x)=\{x\}$ for almost every $x$, thus $\CP^-(x)=\{x\}$ by
construction. Then $\CB = \sigma(\CP^-)$ up
to sets of measure $0$;  by
Lemma~\ref{lemma:constr-excellent} $\imath\imath$)
this contradicts $h_\mu(X,T)>0$.
\end{proof}

\begin{remark}
Although ${\mathbf A}$ is Borel, the set $J$ of points that belong to
a proper asymptotic pair may not be Borel. Nevertheless $J$ is measurable
(modulo null sets) for all Borel measures.
\end{remark}

\begin{remark} As $h(X,T)=h(X,T^{-1})$,  there are also proper asymptotic pairs
for $T^{-1}$. It will be shown later that the stable
classes of $x$ under $T$ and $T^{-1}$ do not coincide.
\end{remark}

\begin{proposition}
  Let $\pi:(X,T)\to(Y,S)$ be a topological factor map, collapsing all proper
asymptotic pairs of $X$. Then $h(Y,S)=0$.
\end{proposition}

\begin{proof}
The difficulty here comes from the fact that the system $Y$ can
have  proper asymptotic pairs \cite{Y}.

Denote the  Borel $\sigma$-algebra of $Y$ by  $\CB_Y$.
Let $\nu$ be an ergodic measure  on $Y$:
  $\nu$ has a preimage $\mu$ under $\pi$, which is $T$-ergodic~\cite{DGS}.
Let $\CP$ be the excellent partition of $(X,\CB,T,\mu)$ constructed in
Lemma~\ref{lemma:constr-excellent}.
When two points belong to the same atom of $\CP^-$
  they are asymptotic: they are collapsed by $\pi$
and belong to the same atom of $\pi^{-1}\CB_Y$.
This means that the $\sigma$-algebra $\pi^{-1}(\CB_Y)$ is contained in the
$\sigma$-algebra $\CP^-$.
As $\pi^{-1}(\CB_Y)$ is invariant by $T$,
it follows from Lemma~\ref{lemma:excellent} that it is contained
in $\Pi_\mu$.
Thus, for any  finite partition $\CQ$ of $Y$,
the partition $\pi^{-1}(\CQ)$ of $X$ is $\Pi_\mu$-measurable and
$$
h_\nu(\CQ,S)=h_\mu(\pi^{-1}(\CQ),T)=0\ .
$$
Therefore $h_\nu(Y,S)=0$; the conclusion follows from the Variational
Principle.
\end{proof}

\subsection{The non-invertible case}
                         \label{subsec:non-invertible}

Let $(X,T)$ be a non-invertible topological system:
$X$ is a compact metric space for the distance $d$,
and $T:X\to X$ is continuous and onto but not one-to-one.

$X$ evidently admits proper asymptotic pairs,
namely any pair $(x,y)$ with $x\neq y$ and $T^nx=T^ny$ for
some $n>0$; when $(X,T)$ is a subshift all asymptotic pairs are of this kind.
It is nevertheless not obvious, and interesting to know, that the
almost-everywhere result of Proposition~\ref{prop:existence} holds in the
non-invertible case too.

\begin{proposition}
                    \label{prop:non-invertible}
Let  $(X,T)$ be a non-invertible topological dynamical system.
The set of points belonging to a proper asymptotic pair
has measure $1$ for any  ergodic measure of positive entropy.
\end{proposition}

\begin{proof}
Recall the definition of the natural
extension $(\tilde X,\tilde T)$ of $(X,T)$:
denote by $\tilde x=(x_n;n\in\Z)$ a point of $X^\Z$, and by
$\tilde X$ the closed subset of $X^{\Z}$ consisting of points
$\tilde x$ such that $x_{n+1}=Tx_n$ for all $n$.
$\tilde X$ is invariant by the shift $\tilde T$,
which is a homeomorphism of $\tilde X$.
Moreover, the map $\pi: \tilde{x}\mapsto x_0$ is onto by compactness
and satisfies $T\circ\pi=\pi\circ\tilde T$.

The  topology of $\tilde X$  is defined by the distance
$$
\tilde d(\tilde x,\tilde y)=\sum_{n\in\Z}2^{-|n|}d(x_n,y_n) \ .
$$
Thus a pair $(\tilde x,\tilde y)$ is asymptotic in $\tilde X$
if and only if the pair $(x_0,y_0)$ is asymptotic in $X$.

Let $J$ be the subset of $X$ consisting of all points belonging to a
proper asymptotic pair,
and let $\tilde J$ have the same definition in $\tilde{X}$. Since $T$ is
onto, $T^{-1}J\subset J$.
Let $z\in \pi(\tilde J)$.
Choose $\tilde x\in\tilde J$ with $x_0=z$, then there exists
$\tilde y\neq\tilde x$ such that $(\tilde x,\tilde y)$ is asymptotic.
There exists $k\geq  0$ such that $y_{-k}\neq x_{-k}$,
thus $(x_{-k},y_{-k})$ is a proper asymptotic pair in $X$,
and $x_{-k}\in J$.
It follows from $z=x_0=T^kx_{-k}$ that $z\in T^kJ$.
Finally $\pi(\tilde J)\subset\bigcup_{k\geq 0}T^kJ$.

Let $\mu$ be an ergodic measure on $X$ with $h_\mu(X,T)>0$.
It lifts to an ergodic measure $\tilde\mu$
on $\tilde X$, with $h_{\tilde\mu}(\tilde{X},\tilde T)>0$.
By Proposition~\ref{prop:existence}, $\tilde\mu(\tilde J)=1$,
thus $\mu(\pi(\tilde J))=1$, which by the inclusion above implies that
$\mu(T^k J)>0$ for some $k$.

For every $k$, $T^{-k}(T^k J)\subset J$: if $T^kx\in T^k J$ there exist
$y\in J$ and $z$ such that $T^k x=T^k y$ and $(y,z)$ is a proper
asymptotic pair. Then either $(x,z)$ or $(x,y)$ is a proper asymptotic
pair depending on whether $x=y$ or not, and $x\in J$.
By the inclusion above it  follows that $\mu(J)>0$, and
since $\mu$ is ergodic $\mu(J)=1$.
\end{proof}

\section{Relatively independent squares}
                               \label{sec:squares}
\subsection{Background}
                          \label{subsec:squares-background}

Let $(X,T)$ be a topological dynamical system, $\CB$ be its Borel
$\sigma$-algebra, and $\mu$ be an ergodic measure.

For the definition and classical properties of conditional
expectations used in this section see~\cite{D1}, \cite{D2},
\cite{B}.
We shall use the

\begin{MT}
Let $(\CG_n)_{n\geq 1}$ be a decreasing sequence of sub-$\sigma$-algebras of
$\CB$ and let $\CG=\bigcap_{n\geq 1}\CG_n$.
  For every $f\in L^2(\mu)$,
$\E(f\mid\CG_n)\to \E(f\mid\CG)$ in $L^2(\mu)$ and almost everywhere.
\end{MT}

The definition of the relatively independent (or conditional)
product of two systems can be found in \cite{R}.

\begin{definition}
                      \label{def:squares}
  Let $\CG$ be a sub-$\sigma$-algebra of $\CB$. The
\emph{conditional square} $\mu\ttimes_\CG\mu$ of $\mu$ relatively to
$\CG$ is the measure on $(X\times X,\CB\otimes\CB)$ determined by
$$
\forall A,B\in\CB,\ \mu\ttimes_\CG\mu(A\times B)=
\int\E(\one_A\mid\CG)(x)\,\E(\one_B\mid\CG)(x)\,d\mu(x)\ .
$$
\end{definition}

$\mu\ttimes_\CG\mu$ is a probability measure, and its two projections
on $X$ are equal to $\mu$.

By standard arguments for every pair of bounded Borel functions $f,g$ on $X$
one has
$$
\int f(x)g(y)\,d(\mu\ttimes_\CG\mu)(x,y)=
\int \E(f\mid\CG)(x)\,\E(g\mid\CG)(x)\,d\mu(x)\ .
$$

The following lemma states the properties  of conditional squares that will be
used.

\begin{lemma}
                   \label{lemma:squares}
Let $\CG$ be a sub-$\sigma$-algebra of $\CB$.
\begin{itemize}

\item[$\imath$)] $\mu\ttimes_\CG\mu$ is concentrated on the diagonal
$\Delta$ of $X\times X$ if and only if the $\sigma$-algebras $\CG$
and $\CB$ are equal up to null sets.

\item[$\imath\imath$)]
If the $\sigma$-algebra $\CG$ is invariant by $T$, then the measure
$\mu\ttimes_\CG\mu$ is invariant by $T\times T$.

\item[$\imath\imath\imath$)] Let $f$ be a bounded $\CG$-measurable function on
$X$. Then $f(x)=f(y)$ for $\mu\ttimes_\CG\mu$-almost
all $(x,y)\in X\times X$.

\item[$\imath\nu$)]
  Let $(\CG_n)_{n\geq 1}$ be a decreasing sequence of
$\sigma$-algebras with $\bigcap_{n\geq 1}\CG_n=\CG$. Then for all
$A,B\in\CB$ one has
$$
\mu\ttimes_\CG\mu(A\times B)=
\lim_{n\to\infty}\mu\ttimes_{\CG_n}\mu(A\times B)\ ,
$$
and the sequence $(\mu\ttimes_{\CG_n}\mu; n\geq 1)$ converges
weakly to $\mu\ttimes_\CG\mu$.
\end{itemize}
\end{lemma}

\begin{proof}
$\imath$) If $\CG=\CB$, then for all $A,B\in\CB$ we have
$\E(\one_A\mid\CG)=\one_A$ and  $\E(\one_B\mid\CG)=\one_B$ $\mu$-a.e.;
then by definition $\mu\ttimes_\CG\mu(A\times B)=\mu(A\cap B)$; the
measure $\mu\ttimes_\CG\mu$ is the image of $\mu$ under the map
$x\mapsto(x,x)$, thus it is concentrated on  $\Delta$.

If $\mu\ttimes_\CG\mu$ is concentrated on $\Delta$, for all
$A\in\CB$ one has $\mu\ttimes_\CG\mu(A\times(X\setminus A))=0$,
that is,
$$
  \int\E(\one_A\mid\CG)(x)\,\E(\one_{X\setminus A}\mid\CG)(x)\,d\mu(x)=0\ ,
$$
thus the product of the two conditional expectations is equal to $0$
a.e.. As the sum of these two functions is equal to $1$,
each of them is equal to $0$ or $1$ a.e.. It follows that
$\E(\one_A\mid\CG)=\one_A$ a.e., and $A$ is measurable with respect
to $\CG$. The $\sigma$-algebras $\CG$
and $\CB$ are equal up to null sets.

$\imath\imath$) Obvious.

$\imath\imath\imath$) By definition
$$
\int f(x)\overline{f(y)}\,d(\mu\ttimes_\CG\mu)(x,y)=
\int |f(x)|^2\,d\mu(x)\text{ because $f$ is $\CG$-measurable
thus }$$
$$\int|f(x)-f(y)|^2\,d(\mu\ttimes_\CG\mu)(x,y)= 0\ .
$$

$\imath\nu$)
When $f$ and $g$ are bounded measurable functions
on $X$, by the Martingale Theorem
\begin{align}
\label{eq:conv}
\int f(x)g(y)\,d(\mu\ttimes_{\CG_n}\mu)(x,y)
   &=\int \E(f\mid\CG_n)(x)\,\E(g\mid\CG_n)(x)\,d\mu(x)\\
    &  \to \int \E(f\mid\CG)(x)\,\E(g\mid\CG)(x)\,d\mu(x)\notag\\
   & =\int f(x)g(y)\,d(\mu\ttimes_\CG\mu)(x,y)\ .\notag
\end{align}
For $f=\one_A$ and $g=\one_B$ this is the first part of $\imath\nu$). The
family
of  continuous functions $F$ on $X\times X$ such that
$$
\int F(x,y)\,d(\mu\ttimes_{\CG_n}\mu)(x,y)\to
  \int F(x,y)\,d(\mu\ttimes_\CG\mu)(x,y)
$$
is a closed subspace of $\CC(X\times X)$. By equation~\eqref{eq:conv}
it contains all functions $f(x)g(y)$ where $f$ and $g$ belong to
$\CC(X)$ and their linear combinations. By density it is equal to $\CC(X\times
X)$, which completes the proof.
\end{proof}

We consider now the case where $\CG$ is associated to a measurable
partition, also denoted by $\CG$.

\begin{lemma}
               \label{lemma:part-squares}
Let $\CG$ be a measurable partition. Then the set
$$
\Delta_\CG=\bigl\{(x,y)\in X\times X; y\in\CG(x)\bigr\}
$$
belongs to $\CB\otimes\CB$, and $\mu\ttimes_\CG\mu$ is concentrated on
this set.
\end{lemma}

\begin{proof}
Let $(\CG_n)_{n\geq 1}$ be an increasing sequence of finite
partitions with $\bigvee_{n\geq 1}\CG_n=\CG$.
Whenever $A,B$ are two distinct atoms of $\CG_n$ it follows
immediately from the definition that $\mu\ttimes_{\CG_n}\mu(A\times B)=0$.
By Lemma~\ref{lemma:squares}~$\imath\nu$),
$\mu\ttimes_\CG\mu(A\times B)=0$.
Thus, for all $n$ the measure
$\mu\ttimes_\CG\mu$ is concentrated on $\Delta_{\CG_n}$.
But the intersection of these sets is $\Delta_\CG$,
and the result follows.
\end{proof}

\subsection{The `construction $\CC$'}
                                   \label{subsec:C}

In the sequel we use several times the following construction,
referred to as \emph{the construction $\CC$}, with the same
notation.

Let $(X,T)$ be a topological dynamical system, $\CB$ be its Borel
$\sigma$-algebra and ${\mathbf A}$ be the set of asymptotic pairs; it is
a Borel subset of $X\times X$, invariant under $T\times T$.

Let $\mu$ be an invariant ergodic measure. Using
Lemma~\ref{lemma:constr-excellent}, choose an excellent partition $\CP$,
such that any pair
of points belonging to the same atom of $\CP^-$ is asymptotic, and put
$\CF=\sigma(\CP^-)$. By Lemma~\ref{lemma:constr-excellent} again
if $h_\mu(X,T)>0$, $\CF$ is not equal to $\CB$ up to $\mu$-null sets.
In the notation of Lemma~\ref{lemma:part-squares}
$$
  \Delta_\CF\subset {\mathbf A}\ .
$$
For every $n\geq 0$ put
$$
  \CF_n=T^{-n}\CF\ \text{ and } \ \nu_n=\mu\ttimes_{\CF_n}\mu \ ;
$$
one has
$$
   \Delta_{\CF_n}=(T\times T)^{-n}\Delta_\CF\subset {\mathbf A}\text{ and }
   \nu_n=(T\times T)^{-n}\nu_0\ ;
$$
  thus $\nu_n$ is concentrated on ${\mathbf A}$.
Moreover, the sequence of sets $(\Delta_{\CF_n})_{n\geq 0}$ is
increasing;
the sequence
$(\CF_n)_{n\geq 0}$ of $\sigma$-algebras is decreasing and its
intersection is equal to $\Pi_\mu$ up to sets of $\mu$-measure $0$ by
Lemma~\ref{lemma:excellent}.

Define
$$
\lambda=\mu\ttimes_{\Pi_\mu}\mu\ .
$$
 From Lemma~\ref{lemma:squares}~$\imath\nu$)
one gets

\begin{corollary}
                       \label{cor:conv}

\begin{itemize}
\item[$\imath$)]
For every $A,B\in\CB$, $\nu_n(A\times B)\to\lambda(A\times B)$
as $n\to\infty$, and the sequence $(\nu_n)_{n\geq 0}$ of measures on
$X\times X$ converges weakly to $\lambda$.

\item[$\imath\imath$)]
For every  closed subset  $F$  of $X\times X$ with
$(T\times T)F\supset F$ one has $\lambda(F)\geq \nu_0(F)$.

\item[$\imath\imath\imath$)]
For every open subset $U$ of $X\times X$ with $(T\times T)U\subset U$
one has $\lambda(U)\leq\nu_0(U)$.
\end{itemize}
\end{corollary}

\begin{proof} $\imath$) Immediate from
Lemma~\ref{lemma:squares}~$\imath\nu$).

\noindent
$\imath\imath$) Since $F$ is closed and $\nu_n\to\lambda$ weakly one has
$$
\lambda(F)\geq \limsup_{n\to\infty}\nu_n(F)\ .
$$
But the sequence $\nu_n(F)=\nu_0\bigl((T\times T)^nF\bigr)$
is increasing and the result follows.

\noindent
$\imath\imath\imath$) Immediate from $\imath\imath$).
\end{proof}

The next result shows that a
2-set partition of positive entropy separates some asymptotic pair.
Significantly, it does the same for some entropy pair \cite{BHM}.

\begin{corollary}
                 \label{cor:part-asympt}
Let $\CQ=(A_1,A_2)$ be a Borel partition with $h_\mu(\CQ,T)>0$. Then there
exists an asymptotic pair $(x_1,x_2)$ with $x_1\in A_1$ and $x_2\in
A_2$.
\end{corollary}

\begin{proof}
If the result is false, then $(A_1\times A_2)\cap {\mathbf A}=\emptyset$, and
$\nu_n(A_1\times A_2)=0$ for all $n$, thus by Corollary 1 $\imath$)
$$
0=\lambda(A_1\times A_2)=
\int\E(\one_{A_1}\mid\Pi_\mu)(x)\,\E(\one_{A_2}\mid\Pi_\mu)(x)\,d\mu(x)\ .
$$
As the two conditional expectations in the integral are non-negative
and have sum equal to $1$, each of them is equal to $0$ or $1$
a.e., which means that the sets $A_1$ and $A_2$ belong to the
$\sigma$-algebra $\Pi_\mu$; thus $h_\mu(Q,T)=0$, which contradicts the
assumption. \end{proof}

\subsection{Application to entropy pairs.}
                      \label{subsec:appl-entropy-pairs}

The definition of entropy pairs of a topological system $(X,T)$ is
given in~\cite{Bl}.
  The set
$E(X,T)$ of entropy pairs is a $T\times T$ invariant subset of
$X\times X$, and $E(X,T)\cup\Delta$ is closed. The system $(X,T)$ has entropy
pairs if and only if its entropy is positive.

The reader should be reminded of the definition of entropy
pairs for an invariant measure $\mu$~\cite{BHM}.
Let $x,y\in X$ with $x\neq y$.
A partition $\CQ=(A,B)$ is said to \emph{separate} $x$ and $y$ if
$x$ belongs to the interior of $A$ and $y$ to the interior of $B$.
$(x,y)$ is said to be an \emph{entropy pair for $\mu$} if for any
partition $\CQ$ separating $x$ and $y$ one has $h_\mu(\CQ,T)>0$.
Call $E_\mu(X,T)$ the set of entropy pairs for $\mu$.
This set is non-empty if and only if $h_\mu(X,T)>0$.

It is shown in~\cite{BGH} that $E(X,T)=\overline{\bigcup_\mu E_\mu(X,T)}$,
where the union is taken over the family of ergodic measures.

Moreover, Glasner shows in~\cite{G} that for any ergodic
measure $\mu$,  $E_\mu(X,T)$ is the set of non-diagonal points in
  the topological support of $\mu\ttimes_{\Pi_\mu}\mu$ (this result
also follows easily from the definition of $\mu\ttimes_{\CF}\mu$ and
Lemma~\ref{lemma:squares}).

\begin{proposition}
               \label{prop:entropy-pairs}
The closure $\overline{\mathbf A}$ of ${\mathbf A}$ in $X\times X$ contains
the set
$E(X,T)$ of entropy pairs.
\end{proposition}

\begin{proof}
Let $\mu$ be an ergodic measure on $X$. In
the notation of the `construction $\CC$', for every $n$, the measure $\nu_n$
is concentrated on the closed set $\overline{\mathbf A}$, and so is the weak
limit $\lambda$ of the sequence $(\nu_n)$.
By Glasner's result $E_\mu(X,T)\subset \overline{\mathbf A}$.
As this is true for any ergodic $\mu$, the result of~\cite{BGH} quoted above
gives the conclusion.
\end{proof}

\begin{corollary}
           \label{cor:K}
If $(X,T)$ admits an invariant measure $\mu$ of full support such
that $(X,\CB,T,\mu)$ is a $K$-system, then asymptotic pairs are dense
in $X\times X$.
\end{corollary}

\begin{proof}
For such a measure $\mu$ the Pinsker $\sigma$-algebra
$\Pi_\mu$ is trivial, it follows that $\lambda=\mu\times\mu$,
its support is $X\times X$, and $E_\mu(X,T)\cup\Delta
=X\times X=\overline{\mathbf A}$.
\end{proof}
In this case $E(X,T)\cup\Delta=X\times X$,
as shown in~\cite{GW} by different means.

\section{Li-Yorke pairs and instability in negative times}
                                   \label{sec:Li-Yorke}

\begin{lemma}
                         \label{lemma:ergodic-square-pinsker}
Let $(X,\CB,T,\mu)$ be an ergodic system, and
$\lambda=\mu\ttimes_{\Pi_\mu}\mu$.
Then $(X\times X,\CB\otimes\CB,T\times T,\lambda)$ is ergodic.
\end{lemma}

\begin{proof}
Assume that $\lambda$ is not ergodic.
According to Theorems 7.5 and 8.2 in \cite{F1}
there exists a  non-trivial isometric extension of $(X,\Pi_\mu,T,\mu)$
(in the measure-theoretic sense)
which is a factor of $(X,\CB,T,\mu)$.
An ergodic isometric extension is a factor of an ergodic group extension
(Theorem 8.2 in \cite{F1}), thus
an ergodic isometric extension of a $0$-entropy system also has
entropy $0$, and this contradicts the  characterization of
$\Pi_\mu$ as the largest factor of $X$ with entropy $0$.
\end{proof}

In the next Proposition $\nu_0$ is defined as in the `construction $\CC$'
above.

\begin{proposition}
                      \label{prop:bizarre}
Let $(X,T)$ be a topological system, $\mu$ an ergodic
measure  of positive entropy and
$$
\delta=\sup\bigl\{ d(x,y); (x,y)\in E_\mu(X,T)\bigr\}>0\ .
$$
For $\nu_0$-almost every pair $(x,y)\in X\times X$ one has
\begin{gather}
       \label{eq:bizarre+}
\lim_{n\to+\infty} d(T^nx,T^ny)=0 \ ;\\
\liminf_{n\to+\infty}d(T^{-n}x,T^{-n}y)=0
\text{ and }
\limsup_{n\to +\infty} d(T^{-n}x,T^{-n}y)\geq\delta \ ;
      \label{eq:bizarre-}
\end{gather}
in particular $(x,y)$ is a Li-Yorke pair for $T^{-1}$.
\end{proposition}

\begin{proof}
  Let $U$ be an open set in $X\times X$, with $\lambda(U)>0$.
  For every $M\geq 0$ we write
$$
U_M=\bigcup_{m\geq M}(T\times T)^mU\ .
$$
$U_M$ is open, and
$(T\times T)U_M=U_{M+1}\subset U_M$.
Moreover, $\lambda(U_M)\geq\lambda(U)>0$. By ergodicity of
$\lambda$, $\lambda(U_M)=1$.
By Corollary~\ref{cor:conv} $\imath\imath\imath$),
$\nu_0(U_M)\geq\lambda(U_M)=1$.
Let
$$
V=\bigcap_{M\geq 0}U_M=\bigcap_{M\geq 0}\bigcup_{m\geq M}
(T\times T)^mU\ .
$$
$V$ is invariant by $T\times T$, and $\nu_0(V)=1$.

For every integer $r>1$, we can cover $\supp(\lambda)$
by a finite number of open balls of radius $1/r$, each of them intersecting
$\supp(\lambda)$. Taking the union of all
these families we obtain a sequence $(U_k)_{k\geq 1}$ of open
sets, with $U_k\cap\supp(\lambda)\neq\emptyset$ for all $k$; each
point of $\supp(\lambda)$ belongs to $U_k$ for infinitely many values
of $k$; the diameter of $U_k$ tends to $0$ as $k\to\infty$. To each $k$
we associate a set $V_k$ as above, and write
$G=\bigcap_{k\geq 1}V_k$. We have $\nu_0(G)=1$.

Let $(x,y)$ be a point in $G$.
For each $k$,  $(T\times T)^{-n}(x,y)\in U_k$ for infinitely many values of
$n$, thus the negative orbit of $(x,y)$ is dense in $\supp(\lambda)$.

By Glasner's result \cite{G}, $\supp(\lambda)=E_\mu(X,T)\cup S(\mu)$, where
$S(\mu)=\{(x,x);x\in\supp(\mu)\}$. Thus  we can choose a pair
$(x_0,y_0)$ in $E_{\mu}(X,T)$ with $d(x_0,y_0)=\delta$ and another pair
$(z_0,z_0)$ in $S(\mu)$. It follows that for all
$(x,y)\in G$ both $(x_0,y_0)$ and $(z_0,z_0)$ are in the closure of the
negative orbit of $(x,y)$, thus
$\limsup_{n\to+\infty}d(T^{-n}x,T^{-n}y)\geq\delta$
and $\liminf_{n\to+\infty}d(T^{-n}x,T^{-n}y)=0$.
Finally, every pair $(x,y)\in G$ satisfies Eq.~\eqref{eq:bizarre-}

Recall that $\nu_0$ is concentrated on $\Delta_\CF$, that is,
$\nu_0(\Delta_\CF)=1$, and that every pair in $\Delta_\CF$ is positively
asymptotic. Thus $\nu_0(\Delta_\CF\cap G)=1$, and every pair in this
set satisfies Eq.~\eqref{eq:bizarre+}.
\end{proof}

\begin{remark}
Assume that $\mu$ is a weakly mixing invariant measure on $X$,
different from a Dirac measure.
Then the same argument as in the proof of Proposition
\ref{prop:bizarre} shows that there exists a $G_\delta$-set $G$ of
$X\times X$,
invariant under $T\times T$, dense in $\supp(\mu)\times\supp(\mu)$,
with $\mu\times\mu(G)=1$
and such that every pair $(x,y)\in G$ is Li-Yorke.
More precisely, there exists $\delta>0$ such that
for every $(x,y)\in G$
$$
\liminf_{n\to +\infty}d(T^nx,T^ny)=0\text{ and }
\limsup_{n\to+\infty} d(T^nx,T^ny)\geq\delta\ .
$$
Here no assumption of positive entropy is needed.
This is related to Iwanik's result on independent sets in topologically weakly
mixing systems   \cite{I}. \end{remark}

\section{There are uncountably many asymptotic pairs}
                                    \label{sec:uncountable}

Up to now most of the results were existence results: we have shown
that a system of positive entropy has asymptotic pairs, and even
pairs which are asymptotic for positive times and Li-Yorke for
negative times. It is interesting to know how large a stable class is, and in
particular whether it can be countable. We prove that the answer is
negative for a.e. class. We need more probabilistic tools.

\subsection{Conditional measures.}
                       \label{subsec:conditional}

Here $X$ is a compact metric space, endowed with its Borel
$\sigma$-algebra $\CB$. Let $\CM(X)$ be the set of probability
measures on $X$, endowed with the topology of weak convergence.  It is
a compact metrizable space.
A proof of the next result can be found in~\cite{F2}.

\begin{lemma}
                     \label{lemma:regular}
Let $\mu$ be a probability measure on $X$, and $\CF$ be a
sub-$\sigma$-algebra of $\CB$. There exists a map
$x\mapsto\mu_x$ from $X$ to $\CM(X)$, measurable with respect to
$\CF$, and such that for every bounded function $f$ on $X$
\begin{equation}
     \label{eq:cond-meas}
\E(f\mid \CF)(x)=\int f(y)\,d\mu_x(y)\ \text{ for } \mu\text{-a.e. } x .
\end{equation}
This map is called a \emph{regular version} of the conditional
probability.
\end{lemma}

  We continue to use the notation of this Lemma.

By definition of the conditional square
(see Sec.~\eqref{subsec:squares-background})
the equality

\begin{equation}
        \label{eq:int-cond-meas}
\mu\ttimes_\CF\mu(K)=\int \mu_x\otimes\mu_x(K)\,d\mu(x)
\end{equation}
holds whenever  $K=A\times B$ where $A,B$ are Borel sets in $X$.
By standard arguments, it holds for every Borel subset $K$ of
$X\times X$. Thus for every bounded Borel function $f$ on
$X\times X$ one has
\begin{equation}\label{eq:int-cond-funct}
\int f(x,y)\,d(\mu\ttimes_\CF\mu)(x,y)
=\int\Bigl(\int f(x,y)\,d\mu_x(y)\Bigr)\,d\mu(x)\ .
\end{equation}

We establish now a condition for the measure $\mu_x$ to be atomless
$\mu$-almost-everywhere.  It is easy to check that the
function $(x,y)\mapsto\mu_x(\{y\})$ is Borel, thus the set
$\{x\in X; \mu_x\text{ is atomless}\}$ is measurable.

\begin{lemma}
       \label{lemma:diagonal}
    Let $\Delta$ be the diagonal of $X\times X$. Then
$\mu\ttimes_\CF\mu(\Delta)=0$ if and only if $\mu_x$ is atomless
for $\mu$-almost all $x\in X$.
\end{lemma}

\begin{proof}
  We write $\nu=\mu\ttimes_\CF\mu$.

By Fubini's Theorem and Eq.~\eqref{eq:int-cond-meas},
$$
\nu(\Delta)=\int \mu_x\otimes\mu_x(\Delta)\,d\mu(x)
=\int\mu_x(\{x\})\,d\mu(x)\ .
$$
The `if' part of the Lemma is now immediate. Now assume that
$\nu(\Delta)=0$. One has
\begin{equation}
     \label{eq:mux0}
\mu_x(\{x\})=0\text{ for $\mu$-almost all }x\ .
\end{equation}

  As the map $x\mapsto\mu_x$ is $\CF$-measurable,
it follows from Lemma~\ref{lemma:squares}~$\imath\imath\imath$)
  that $\mu_x=\mu_y$ for $\nu$-almost all
$(x,y)$, thus
\begin{equation}
     \label{eq:eqmuxy}
\mu_x(\{x\})=\mu_y(\{x\}) \text{ for $\nu$-almost all }(x,y)\ .
\end{equation}
As the first projection of $\nu$ on $X$ is $\mu$, it follows
from Eqs.~\eqref{eq:mux0} and~\eqref{eq:eqmuxy} that
\begin{equation}
         \label{eq:muxy0}
  \mu_y(\{x\})=0 \text{ for $\nu$-almost all }(x,y)\ .
\end{equation}
Using Eq.~\eqref{eq:int-cond-funct} with $f(x,y)=\mu_y(\{x\})$ one gets
$$
  0=\int \mu_y(\{x\})\,d\nu(y)=
  \int\Bigl(\int \mu_y(\{x\})\,d\mu_y(x)\Bigr)\,d\mu(y)\ .
$$
Hence $\int\mu_y(\{x\})\,d\mu_y(x)=0$ for $\mu$-almost all $y$.

But for all $y$ the measure $\mu_y$ is larger than its discrete part
$\tau_y=\sum_z\mu_y(\{z\})\delta_z$, where $\delta_z$ is the Dirac
mass at $z$, and for $\mu$-almost all $y$ we have
$$
0=\int\mu_y(\{x\})\,d\tau_y(x)=\sum_z(\mu_y(\{z\}))^2
$$
thus $\mu_y(\{z\})=0$ for all $z$ and $\mu_y$ is atomless.
\end{proof}

\subsection{Application to asymptotic pairs.}
The next result is a topological counterpart of
Proposition~\ref{prop:bizarre}.

\begin{proposition}
     \label{prop:ouf}
Assume that $h(X,T)>0$. There exist $\delta>0$,
an uncountable subset $F$
of $X$, and for every $x\in F$ an uncountable subset $F_x$ of
$X$ such that for every $y\in F_x$ the relations~\eqref{eq:bizarre+}
and \eqref{eq:bizarre-} hold, that is,
\begin{gather*}
\lim_{n\to+\infty} d(T^nx,T^ny)=0 \ ;\\
\liminf_{n\to+\infty}d(T^{-n}x,T^{-n}y)=0
\text{ and }
\limsup_{n\to +\infty} d(T^{-n}x,T^{-n}y)\geq\delta \ .
\end{gather*}
Here one can choose any $\delta$ such that
$0<\delta<\sup\{d(x,y);(x,y)\in E(X,T)\}$.
\end{proposition}

\begin{proof}Let $D=\sup\{d(x,y);(x,y)\in E(X,T)\}$.
By compactness there exists $(x',y') \in E(X,T)$ with $d(x',y')=D$. Thus for
$0<\delta<D$  there exist an ergodic measure
$\mu$ and a $\mu$-entropy pair $(x_0,y_0)$ close to $(x',y')$ such that
$d(x_0,y_0)\geq \delta$.
Recall that in the notations of the `construction $\CC$'
$\nu_0=\mu\ttimes_\CF\mu$; denote by
$\mu_x$ a regular version of
the conditional probability given $\CF$,
as in Lemma~\ref{lemma:regular}.

\begin{lemma}
           \label{lemma:atomless}
With the assumptions of Proposition~\ref{prop:ouf} for $\mu$-almost every
$x$ the
measure $\mu_x$ is atomless. \end{lemma}

\begin{proof} (of the Lemma)
Assume that the conclusion does not hold.
By Lemma~\ref{lemma:diagonal}, $\nu_0(\Delta)>0$.
As $\Delta$ is invariant under $T\times T$,
by Corollary~\ref{cor:conv}~$\imath\imath$)
$\lambda(\Delta)\geq\nu_0(\Delta)>0$.
By Lemma~\ref{lemma:ergodic-square-pinsker},
$\lambda$ is ergodic for $T\times T$,
thus $\lambda(\Delta)=1$.

By Lemma \ref{lemma:squares} $\imath$)
it means that $\Pi_\mu=\CB$ up to $\mu$-null sets,
thus $h_\mu(X,T)=0$. This is impossible because there exists an
entropy pair $(x_0,y_0)$ for $\mu$.
\end{proof}

We continue the proof of Proposition~\ref{prop:ouf}.
By Proposition~\ref{prop:bizarre}, the
relations~\eqref{eq:bizarre+} and \eqref{eq:bizarre-} hold for
$\nu_0$-almost every $(x,y)\in X\times X$.

For $x\in X$, let  $F_x$ be the set of all points $y\in X$ such that these
relations hold for $(x,y)$. Since $\nu_0=\mu\ttimes_\CF\mu$ one has
$\int\mu_x(F_x)\,d\mu(x)=1$, thus $\mu_x(F_x)=1$ for $\mu$-almost
all $x$. Let
$$F=
\{x\in X;\mu_x(F_x)=1\}\cap\{x\in X;\mu_x\text{ is atomless}\}\ .
$$
Then $\mu(F)=1$.
The measure $\mu$ is ergodic and of positive entropy, thus atomless.
Hence the set $F$ is uncountable.
For $x\in F$, $\mu_x(F_x)=1$ and $\mu_x$ is atomless, therefore
$F_x$ is an uncountable set.
\end{proof}

%***************************************************************************

\end{document}